\documentclass[11pt,reqno,oneside]{amsart}

\usepackage{graphicx}
\usepackage{amssymb}
\usepackage{epstopdf}
\usepackage{bbm}
\usepackage{amsfonts}
\usepackage[T1]{fontenc}
\usepackage{latexsym,amssymb,amsmath,amsfonts}
\usepackage{graphicx}
\usepackage{subfigure}
\usepackage{epsf}
\usepackage{epstopdf}
\usepackage{amssymb}
\usepackage{hyperref}
\usepackage{comment}

\usepackage[a4paper, total={6in, 9in}]{geometry}

\usepackage{amsmath,amsfonts,amsthm,mathrsfs,amssymb,cite}
\usepackage[usenames]{color}

\newtheorem{thm}{Theorem}[section]

\newtheorem{prop}{Proposition}[section]
\theoremstyle{definition}
\newtheorem{defn}{Definition}[section]

\newtheorem{exmp}[thm]{Example}

\theoremstyle{remark}

\numberwithin{equation}{section}

\numberwithin{equation}{section}

\newcounter{saveeqn}

\newcommand{\be}{\begin{eqnarray}}
\newcommand{\ben}{\begin{eqnarray*}}
\newcommand{\en}{\end{eqnarray}}
\newcommand{\enn}{\end{eqnarray*}}

\title[Symmetric-adjoint and Symplectic-adjoint Methods and Applications]{Symmetric-adjoint and Symplectic-adjoint Methods and Their Applications }

\author{Geng Sun}
\address{Institute of Mathematics, Chinese Academy of Sciences, Beijing, China}
\email{sung@amss.ac.cn}

\author{Siqing Gan}
\address{School of Mathematics and Statistics, Central South University, Changsha, China}
\email{sqgan@csu.edu.cn}

\author{Hongyu Liu}
\address{Department of Mathematics, Hong Kong Baptist University, Kowloon, Hong Kong, China}
\email{hongyu.liuip@gmail.com}

\author{Zaijiu Shang}
\address{Institute of Mathematics, Chinese Academy of Sciences, Beijing, China}
\email{zaijiu@amss.ac.cn}

\date{} 
\begin{document}
\maketitle

\begin{abstract}

Symmetric method and symplectic method are classical notions in the theory of Runge-Kutta methods. They can generate numerical flows that respectively preserve the symmetry and symplecticity of the continuous flows in the phase space. Adjoint method is an important way of constructing a new Runge-Kutta method via the symmetrisation of another Runge-Kutta method. In this paper, we introduce a new notion, called {\it symplectic-adjoint} Runge-Kutta method. We prove some interesting properties of the symmetric-adjoint and symplectic-adjoint methods. These properties reveal some intrinsic connections among several classical classes of Runge-Kutta methods. In particular, the newly introduced notion and the corresponding properties enable us to develop a novel and practical approach of constructing high-order explicit Runge-Kutta methods, which is a challenging and longly overlooked topic in the theory of Runge-Kutta methods.

\medskip


\noindent{\bf 2010 Mathematics Subject Classification:}~~Runge-Kutta method; symmetric and symplectic; adjoint; high-order; explicit method

\end{abstract}

\section{Introduction}

In his book {\it Institutionum calculi integralis} in 1768 \cite{Euler68}, L. Euler introduced a first-order numerical procedure for solving ordinary differential equations (ODEs), which is nowadays known as the {\it Euler method}. It is the most basic explicit method for the numerical integration of initial value problems for ODEs of the form
\begin{equation}\label{eq:ode1}
y'(t)=f(t, y),\quad y(t_0)=y_0.
\end{equation}
More than 100 years later around 1900, the German mathematicians C. Runge and M. W. Kutta developed the nowadays known Runge-Kutta methods. The Runge-Kutta methods are a family of iterative methods for the numerical integration of \eqref{eq:ode1}, and they include the Euler method as a simple and special case. Choosing a step-size $h\in\mathbb{R}_+$, an $s$-stage Runge-Kutta method takes the following form,
\begin{equation}\label{eq:RK1}
y_{n+1}=y_n+h\sum_{i=1}^s b_i k_i,
\end{equation}
where
\begin{equation}\label{eq:RK2}
k_i=f(t_n+c_i h, y_n+h\sum_{j=1}^s a_{ij} k_j), \quad i=1,2,\ldots, s.
\end{equation}
The matrix $\mathbf{A}=(a_{ij})_{i,j=1}^s$ is called the Runge-Kutta matrix, while $\mathbf{b}=(b_i)_{i=1}^s$ and $\mathbf{c}=(c_i)_{i=1}^s$ are known as the weighting and nodal vectors, respectively. These data are usually arranged in a mnemonic device, known as a Buthcer tableau (after J. C. Butcher),
\begin{equation}\label{eq:RK3}
\begin{tabular}{c|cccc}
$c_1$ & $a_{11}$ & $a_{12}$ & $\cdots$ & $a_{1s}$\\
$c_2$ & $a_{21}$ & $a_{22}$ & $\cdots$ & $a_{2s}$\\
$\vdots$ & $\vdots$ & $\vdots$ & $\ddots$ & $\vdots$\\
$c_s$ & $a_{s1}$ & $a_{s2}$ & $\cdots$ & $a_{ss}$\\
\hline
& $b_1$ & $b_2$ & $\cdots$ & $b_s$
\end{tabular}
= \begin{tabular}{c|c}
$\mathbf{c}$ & $\mathbf{A}$\\
\hline
&$\mathbf{b}^T$
\end{tabular}\ \ .
\end{equation}

The study of Runge-Kutta methods has a long and coloured history. In this paper, we are concerned with several classical notions in the theory of Runge-Kutta methods including symmetric method, symplectic method, adjoint method, and explicit/implicit method. Symmetric method can generate numerical flows that preserve the symmetry of the continuous flows in the phase space. Adjoint method is an important way of constructing a new Runge-Kutta method via the symmetrisation of another Runge-Kutta method. It was originated by Scherer \cite{Scherer79} and Butcher \cite{Butcher87} who proposed the notion of reflected Runge-Kutta methods and in \cite{Hairer83, Hairer02}, this class of Runge-Kutta methods is referred to as the adjoint methods, and their properties were studied. In this paper, motived by the symmetric-adjoint method, we introduce a new notion of {\it symplectic-adjoint method} of a Runge-Kutta method. Symplectic method preserves the symplectic structure of a Hamiltonian flow and, it was pioneered by K. Feng \cite{Feng84} and has inspired a flourishing branch of numerical mathematics called {\it geometric numerical integration} (cf. \cite{Hairer02}). We discover some novel and interesting properties about the symmetric-adjoint and symplectic-adjoint methods. On the one hand, those properties are of independent mathematical interest and they actually reveal some beautiful and intriguing connections among several classes of classical Runge-Kutta methods. On the other hand, they can be used to significantly simplify the order conditions of Runge-Kutta methods and pave the way for the practical construction of high-order explicit methods.

For a Runge-Kutta method of the form \eqref{eq:RK3}, if $\mathbf{A}$ is a strictly lower-triangular matrix, namely $a_{ij}=0$ for $i\leq j$, then the Runge-Kutta method \eqref{eq:RK3} is called {\it explicit}, otherwise it is called {\it implicit}. It is widely known that implicit Runge-Kutta methods possess many fine properties which make them particularly suitable for the numerical integration of stiff ODEs and Hamiltonian systems. There are many research monographs on the implicit Runge-Kutta methods including systematic frameworks for the constructions and the corresponding properties \cite{Butcher87,Hairer83, Hairer91, Hairer02, Stetter73}. There are three classical classes of implicit Runge-Kutta methods, including the Gauss-type, Radau-type and Lobatto-type methods. The Gauss-type methods are of stage $s$ and order $2s$, and they are the unique class of methods that possess such a stage-order relationship \cite{Butcher642,Ceschino66}. The Radau-type and Lobatto-type methods are, respectively, of stage $s$ and orders $2s-1$ and $2s-2$. However, the theory and study of explicit Runge-Kutta methods are not as rich as the implicit ones. The most widely known member of the Runge-Kutta family is generally referred to as ``RK4'', which is an explicit method of stage $4$ and order $4$. Kutta tried but failed to construct a $5$-stage and $5$-order explicit method, and then he managed to construct a $6$-stage and $5$-order explicit method \cite{Kutta}. There was an error in Kutta's construction and was later fixed by Nystr\"om \cite{Nystrom25}. It is until 1960's that it was proved that there does not exist explicit Runge-Kutta method which is of order $5$ and stage $5$ \cite{Butcher64,Butcher65,Ceschino66,Shanks66}. In addition, Butcher constructed explicit methods that are of stage $7$ and order $6$, and he proved that there does not exist explicit method of stage $8$ and order $7$, and further claimed the existence of such methods in stage $9$ and order $7$, which was indeed constructed by Verner \cite{Verner71}. In 1985, Butcher \cite{Butcher85} proved the nonexistence of explicit Runge-Kutta method of stage $10$ and order $8$, which is known as the Butcher's order barrier. There is also some significant progress by the other researchers on the construction of explicit Runge-Kutta methods:  Curtis \cite{Curtis70} and Copper \& Verner \cite{Cooper72} constructed explicit methods of stage 11 and order 8; Curtis \cite{Curtis75} constructed explicit method of stage 18 and order 10; Hairer \cite{Hairer78} obtained explicit method of stage 17 and order 10. Since then, there is very little progress on the construction of high-order explicit Runge-Kutta methods. This is in sharp contrast from that for implicit methods and indeed according to our earlier discussion, one can easily obtain an implicit method of an arbitrary order. The major difficulty of constructing high-order explicit methods comes from the enormous number of order conditions as the stage and order of the Runge-Kutta method increase. In Table~1, we list the involving numbers, where $p$ stands for the order, $s$ stands for the stage, $m$ stands for the number of entries of an explicit Runge-Kutta method, and $N$ signifies the number of order conditions imposed on those $m$ entries. It can be seen that if one intends to construct a high-order explicit method, one needs to deal with a huge number of (nonlinear) algebraic order conditions satisfied by its entries (in a comparably much small number). Solving the aforesaid extremely over-determined nonlinear system would be fraught with immense difficulties, and it is a major reason that the study of constructing high-order explicit Runge-Kutta methods has been longly overlooked since the earlier mentioned efforts. 
\begin{table}[t]
\centering
\caption{}\label{tab:aStrangeTable}
\begin{tabular}{|c|c|c|c|c|c|c|c|c|c|c|c|c|c|c|}
\hline $p$ &1&2&3&4&5&6&7&8&9&10&11&12&13&14\\
\hline $s$ &1&2&3&4&6&7&9&11& & & & & &\\
\hline $m$ &1&3&6&10&21&28&45&66& & & & & & \\
\hline  $N$& 1&2&4&8&17&37&85&200&486&1205&3047&7813&20300&53264\\
\hline
\end{tabular}
\end{table}

It is one of the main aims of the current article to develop a much feasible framework in constructing high-order explicit Runge-Kutta methods. By carefully analyzing those order conditions, we discover some common features that can significantly reduce the redundancies. In fact, it is such an observation that motivates us to introduce the notion of symplectic-adjoint method for a Runge-Kutta method. This together with the classical symmetric-adjoint method provides the right tools for simplifying the order conditions to a manageable level. However, as remarked earlier, our study of the symplectic-adjoint and symmetric-adjoint methods is of significant mathematical values for its own sake.

The rest of the paper is organized as follows. In Section 2, we present some preliminary results on the order conditions, symmetric and symplectic methods. Section 3 is devoted to the analysis of symmetric-adjoint and symplectic-adjoint methods. In Section 4, we consider the construction of high-order explicit Runge-Kutta methods and construct a class of explicit methods in stage 6 and order 5.

\section{Order conditions, symmetric and symplectic Runge-Kutta methods}

Associated with the Runge-Kutta method in \eqref{eq:RK3}, we introduce the following so-called simplified order conditions \cite{Butcher642},
\begin{align}
B(p):& \sum_{i=1}^{s}b_ic_i^{k-1}=\frac{1}{k}, k=1,...,p;\label{eq:order1}\\
C(\eta):& \sum_{j=1}^{s}a_{ij}c_j^{k-1}=\frac{1}{k}c_i^k, i=1,...,s, k=1,...,\eta;\label{eq:order2}\\
D(\zeta):& \sum_{i=1}^{s}b_ic_i^{k-1}a_{ij}=\frac{1}{k}b_j(1-c_j^k), j=1,...,s, k=1,...,\zeta. \label{eq:order3}
\end{align}
There holds the following result (cf. \cite{Butcher642,Hairer83}),

\begin{thm}\label{thm:1}
If the coefficients of a Runge-Kutta method of the form \eqref{eq:RK3} satisfies $B(p), C(\eta)$ and $D(\zeta)$ with $p\leq \eta+\zeta+1$ and $p\leq 2\eta+2$, then the method is of order $p$.
\end{thm}

By Theorem~\ref{thm:1}, for the construction of a Runge-Kutta method with a specific order of accuracy, it suffices to consider those order conditions \eqref{eq:order1}--\eqref{eq:order3} for its coefficients. However, it is noted that an explicit Runge-Kutta method at most satisfies $C(1)$ and $D(1)$, and hence Theorem~\ref{thm:1} does not apply to the construction of high-order explicit Runge-Kutta methods.

Next, we introduce the so-called test equation,
\begin{equation}\label{eq:test1}
y'=\lambda y,\quad y(0)=y_0\quad \mbox{with\ \ $\lambda\in\mathbb{C}$ and $\Re\lambda\leq 0$}.
\end{equation}
By applying the Runge-Kutta method $(\mathbf{A}, \mathbf{b},\mathbf{c})$ \eqref{eq:RK3} to numerically solve the test equation \eqref{eq:test1} with a step-size $h\in\mathbb{R}_+$, one has that
\begin{equation}\label{eq:sf1}
y_{n+1}=R(z) y_n,\quad z=\lambda h\in\mathbb{C},
\end{equation}
where
\begin{equation}\label{eq:sf2}
R(z)=\frac{\mathrm{det}(I+z(\mathbf{e}\mathbf{b}^T-\mathbf{A}))}{\mathrm{det}(I-z\mathbf{A})},
\end{equation}
with $\mathbf{e}=(1,1,\ldots,1)\in\mathbb{R}^s$, is known as the stability function for the Runge-Kutta method. Then the set $S:=\{z\in\mathbb{C}; |R(z)|\leq 1\}$ is called the stability region of the Runge-Kutta method, and if $S\supset\mathbb{C}^-$, then the method is called $A$-stable. It is known that explicit Runge-Kutta methods are not $A$-stable (cf. \cite{Hairer91}). If an $A$-stable Runge-Kutta method further satisfies that $\lim_{|z|\rightarrow+\infty} R(z)=0$, then the method is called $L$-stable. We also recall the algebraic stability of a Runge-Kutta method $(\mathbf{A},\mathbf{b}, \mathbf{c})$ if its coefficients satisfy the following two conditions (cf. \cite{Burrage79, Crouzeix79}):
\begin{enumerate}
\item $b_i\geq 0,\ \ \ i=1,2,\ldots, s$;

\item The matrix $M:=(m_{ij})_{i,j=1}^s$ is nonnegative, where
\begin{equation}\label{eq:M}
m_{ij}=b_ia_{ij}+b_ja_{ji}-b_ib_j,\ \ i,j=1,2,\ldots, s.
\end{equation}
\end{enumerate}
It is known that algebraic stability implies $A$-stability for a Runge-Kutta method.

In what follows, we set
\begin{equation}\label{eq:n1}
C(q)_i=\sum_{j=1}^{s}a_{ij}c_j^{q-1}-\frac{1}{q}c_i^q,\quad i=1,...,s,\quad q\in\mathbb{N}.
\end{equation}
Then for any $q\in\mathbb{N}$, $\mathcal{C}_q:=(C(q)_i)_{i=1}^s\in\mathbb{R}^s$, and the simplified order condition $C(\eta)$ actually means that $\mathcal{C}_q=0$ for $q=1,2,\ldots,\eta$. In a similar manner, we introduce $D(q)_i$ and $\mathcal{D}_q$. We have
\begin{prop}\label{prop:1}
Suppose that the order conditions $B(p), C(\eta)$ and $D(\zeta)$ hold for $\eta, \zeta< p$. Then for any given $q\leq \eta$, there exists $i_q$ such that $C(q)_{i_q}$ can be expressed in terms of $\{C(q)_i\}_{i=1, i\neq i_q}^s$. Similarly, for any given $r\leq \zeta$, there exists $i_r$ such that $D(r)_{i_r}$ can be expressed in terms of $\{D(r)_i\}_{i=1, i\neq i_r}^s$.
\end{prop}

\begin{proof}
The proof follows directly from the following two identities,
\begin{equation*}
\begin{aligned}
\sum_{i=1}^{s}b_iC(q)_i&=\sum_{i=1}^sb_i(\sum_{j=1}^{s}c_j^{q-1}a_{ij}-\frac{1}{q}c_i^q)\\
&=\sum_{j=1}^{s}(\sum_{i=1}^{s}b_ia_{ij})c_j^{q-1}-\frac{1}{q}\sum_{i=1}^{s}b_ic_i^q\\
&=\sum_{j=1}^sb_j(1-c_j)c_j^{q-1}-\frac{1}{q(q+1)}\\
&=(\frac{1}{q}-\frac{1}{q+1})-\frac{1}{q(q+1)}=0,
\end{aligned}
\end{equation*}
and
\begin{equation*}
\begin{aligned}
\sum_{j=1}^{s}D(r)_j&=\sum_{j=1}^{s}(\sum_{i=1}^{s}b_ic_i^{r-1}a_{ij}-\frac{1}{r}b_j(1-c_j^r))\\
&=\sum_{i=1}^{s}b_ic_i^r-\frac{1}{r}(1-\frac{1}{r+1})\\
&=\frac{1}{r+1}-\frac{1}{r+1}=0.
\end{aligned}
\end{equation*}
\end{proof}

Next, we introduce the symmetric Runge-Kutta method and its related order conditions.
\begin{defn}\label{def:symmetric}
A Runge-Kutta method is said to be symmetric if its stability function satisfies $\Phi^{-1}=\Phi$.
\end{defn}

\begin{thm}[\cite{Stetter73, Wanner73}]\label{thm:symmetric1}
For an $s$-stage Runge-Kutta method $(\mathbf{A, b, c})$, if its coefficients satisfy the following relationships,
\begin{equation}
a_{ij}=b_{s+1-j}-a_{s+1-i,s+1-j}, \quad b_j=b_{s+1-j}\quad i,j=1,...,s,
\end{equation}
then the method is symmetric. Moreover, explicit Runge-Kutta method are not symmetric.
\end{thm}

\begin{prop}\label{prop:2}
Consider an $s$-stage symmetric Runge-Kutta method $(\mathbf{A, b, c})$. Suppose the order condition $B(p), C(\eta)$ and $D(\zeta)$ are fulfilled for this method and $\eta, \zeta<p$. Then we have that $C(q)_i=0$ implies $C(q)_{s+1-i}=0$ and vice versa, for $i=1,2,\ldots, s$ and $q=1,2,\ldots, \eta$. Similarly, $D(q)_i=0$ implies $D(q)_{s+1-i}=0$ and vice versa, for $i=1,2,\ldots, s$ and $q=1,2,\ldots,\zeta$.
\end{prop}

\begin{proof}
We first prove that $C(q)_i=0$ is equivalent to $C(q)_{s+1-i}=0$. To that end, we have by direct calculations,
\begin{equation}\label{eq:d1}
\begin{split}
&\sum_{j=1}^{s}a_{ij}c_j^{q-1}\\
=& \sum_{j=1}^s(b_{s+1-j}-a_{s+1-i,s+1-j})c_j^{q-1}=\frac{1}{q}-\sum_{j=1}^sa_{s+1-i,s+1-j}(1-c_{s+1-j})^{q-1}\\
=& \frac{1}{q}-\sum_{j=1}^sa_{s+1-i,s+1-j}\sum_{l=0}^{q-1}(-1)^l\dbinom{l}{q-1}c_{s+1-j}^l\\
=& \frac{1}{q}-(-1)^{q-1}\sum_{j=1}^sa_{s+1-i,s+1-j}c_{s+1-j}^{q-1}-\sum_{j=1}^sa_{s+1-i,s+1-j}\sum_{l=0}^{q-2}(-1)^{l}\dbinom{l}{q-1}c_{s+1-j}^l,
\end{split}
\end{equation}
and
\begin{equation}\label{eq:d2}
\frac{1}{q}c_i^q=\frac{1}{q}(1-c_{s+1-i})^q=\frac{1}{q}\sum_{l=0}^q(-1)^l\dbinom{l}{q}c_{s+1-i}^l
=\frac{1}{q}\left(1+\sum_{l=1}^q(-1)^lc_{s+1-i}^l\dbinom{l}{q}\right).
\end{equation}
Using \eqref{eq:d1}, \eqref{eq:d2} and $C(q)_i=0$, one can further deduce that
\begin{equation*}
\begin{aligned}
&(-1)^q\sum_{j=1}^sa_{s+1-i,s+1-j}c_{s+1-j}^{q-1}\\
&=\frac{1}{q}\sum_{l=1}^q(-1)^l\dbinom{l}{q}c_{s+1-i}^l+\sum_{j=1}^sa_{s+1-i,s+1-j}\sum_{l=0}^{q-2}(-1)^l\dbinom{l}{q-1}c_{s+1-j}^l\\
&=\frac{1}{q}\sum_{l=1}^q(-1)^l\dbinom{l}{q}c_{s+1-i}^l+\sum_{l=0}^{q-2}(-1)^l\dbinom{l}{q-1}\frac{c_{s+1-i}^{l+1}}{l+1}\\
&=\frac{1}{q}(-1)^qc_{s+1-i}^{q}+\frac{1}{q}\sum_{l=1}^{q-1}(-1)^l\dbinom{l}{q}c_{s+1-i}^{l}+\sum_{l=0}^{q-2}(-1)^l\frac{1}{l+1}\dbinom{l}{q-1}c_{s+1-i}^{l+1}\\
&=\frac{1}{q}(-1)^qc_{s+1-i}^{q}+\frac{1}{q}\sum_{l=1}^{q-1}(-1)^l\dbinom{l}{q}c_{s+1-i}^{l}+\frac{1}{q}\sum_{l=0}^{q-2}(-1)^l\dbinom{l+1}{q}c_{s+1-i}^{l+1}\\
&=\frac{1}{q}(-1)^qc_{s+1-i}^{q}+\frac{1}{q}\sum_{l=0}^{q-2}(-1)^{l+1}\dbinom{l+1}{q}c_{s+1-i}^{l+1}+\frac{1}{q}\sum_{l=0}^{q-2}(-1)^l\dbinom{l+1}{q}c_{s+1-i}^{l+1}\\
&=\frac{1}{q}(-1)^qc_{s+1-i}^{q},
\end{aligned}
\end{equation*}
which readily gives that
\begin{equation*}
\sum_{j=1}^sa_{s+1-i,s+1-j}c_{s+1-j}^{q-1}=\frac{1}{q}c_{s+1-i}^{q}\Rightarrow\sum_{j=1}^sa_{s+1-i,j}c_j^{q-1}=\frac{1}{q}c_{s+1-i}^{q}.
\end{equation*}

Next, we prove that $D(q)_i=0$ is equivalent to $D(q)_{s+1-i}=0$. First, we have by direct calculations that
\begin{equation}\label{eq:dd1}
\begin{aligned}
\sum_{j=1}^sb_jc_j^{q-1}a_{ji}
&=\sum_{j=1}^sb_jc_j^{q-1}(b_{s+1-i}-a_{s+1-j,s+1-i})\\
&=\frac{1}{q}b_{s+1-i}-\sum_{j=1}^sb_{s+1-j}(1-c_{s+1-j})^{q-1}a_{s+1-j,s+1-i}\\
&=\frac{1}{q}b_{s+1-i}-\sum_{j=1}^sb_{s+1-j}a_{s+1-j,s+1-i}\sum_{k=0}^{q-1}(-1)^k\dbinom{k}{q-1}c_{s+1-j}^k\\
&=\frac{1}{q}b_{s+1-i}-(-1)^{q-1}\sum_{j=1}^sb_{s+1-j}a_{s+1-j,s+1-i}c_{s+1-j}^{q-1}\\
&-\sum_{j=1}^sb_{s+1-j}a_{s+1-j,s+1-i}\sum_{k=0}^{q-2}(-1)^k\dbinom{k}{q-1}c_{s+1-j}^k\\
\end{aligned}
\end{equation}
and
\begin{equation}\label{eq:dd2}
b_i(1-c_i^q)\frac{1}{q}=\frac{1}{q}b_{s+1-i}(1-(1-c_{s+1-i})^q)=\frac{1}{q}b_{s+1-i}-\frac{1}{q}b_{s+1-i}\sum_{k=0}^q(-1)^k\dbinom{k}{q}c_{s+1-i}^k.
\end{equation}
Using \eqref{eq:dd1}, \eqref{eq:dd2} and $D(q)_i=0$, one can further deduce that
\begin{equation*}
\begin{aligned}
&(-1)^q\sum_{j=1}^sb_{s+1-j}a_{s+1-j,s+1-i}c_{s+1-j}^{q-1}\\
&=\sum_{j=1}^sb_{s+1-j}a_{s+1-j,s+1-i}\sum_{k=0}^{q-2}(-1)^k\dbinom{k}{q-1}c_{s+1-j}^k-\frac{1}{q}b_{s+1-i}\sum_{k=0}^q(-1)^k\dbinom{k}{q}c_{s+1-i}^k\\
&=\sum_{k=0}^{q-2}(-1)^k\dbinom{k}{q-1}b_{s+1-i}\frac{(1-c_{s+1-i}^{k+1})}{k+1}-\frac{1}{q}(-1)^qb_{s+1-i}c_{s+1-i}^{q}\\
&\ \ -\frac{1}{q}b_{s+1-i}\sum_{k=1}^{q-1}(-1)^k\dbinom{k}{q}c_{s+1-i}^k-\frac{1}{q}b_{s+1-i}\\
&=\frac{1}{q}b_{s+1-i}\sum_{k=0}^{q-2}(-1)^k\dbinom{k+1}{q}(1-c_{s+1-i}^{k+1})\\
  &\quad-\frac{1}{q}(-1)^qb_{s+1-i}c_{s+1-i}^{q}-\frac{1}{q}b_{s+1-i}\sum_{k=0}^{q-2}(-1)^{k+1}\dbinom{k+1}{q}c_{s+1-i}^{k+1}-\frac{1}{q}b_{s+1-i}\\
&=\frac{1}{q}b_{s+1-i}\sum_{k=0}^{q-2}(-1)^k\dbinom{k+1}{q}-\frac{1}{q}(-1)^qb_{s+1-i}c_{s+1-i}^{q}-\frac{1}{q}b_{s+1-i}\\
&=\frac{1}{q}\left\{\sum_{k=0}^{q-2}(-1)^k\dbinom{k+1}{q}-(-1)^q\dbinom{q}{q}-1\right\}b_{s+1-i}\\
&\ \ +\frac{1}{q}(-1)^qb_{s+1-i}-\frac{1}{q}(-1)^qb_{s+1-i}c_{s+1-i}^{q}\\
&=-\frac{1}{q}b_{s+1-i}(1-1)^q+\frac{1}{q}(-1)^q b_{s+1-i}-\frac{1}{q}(-1)^qb_{s+1-i}c_{s+1-i}^{q}\\
&=\frac{1}{q}(-1)^qb_{s+1-i}(1-c_{s+1-i}^q)
\Rightarrow \sum_{j=1}^sb_{s+1-j}c_{s+1-j}^{q-1}a_{s+1-j,s+1-i}=\frac{1}{q}b_{s+1-i}(1-c_{s+1-i}^q).
\end{aligned}
\end{equation*}

The proof is complete.
\end{proof}

Next, we consider the symplectic Runge-Kutta method for Hamiltonian systems. Let $p(t)\in\mathbb{R}^N$ and $q(t)\in\mathbb{R}^N$, respectively, denote that generalised momentum and position coordinates, where $t$ is the temporal variable and $N$ is the dimension. Let $H(p, q)$ be a scalar function, signifying the Hamiltonian. The Hamiltonian system is the following evaluation equation,
\begin{equation}\label{eq:Ham1}
p'=-\nabla_q H(p, q),\quad q'=\nabla_p H(p, q).
\end{equation}
The time evolution of Hamilton's equations is a symplectomorphism, meaning that it conserves the symplectic two-form ${\displaystyle dp\wedge dq}$. A numerical scheme is a symplectic integrator if it also conserves this two-form.

\begin{thm}[\cite{Lasagni88,Sanz88, Suris89}]\label{thm:symp1}
An $s$-stage Runge-Kutta method $(\mathbf{A, b, c})$ is symplectic if its coefficients satisfy
\begin{equation}\label{eq:ss0}
b_ia_{ij}+b_ja_{ji}-b_ib_j=0, \quad i,j=1,...,s,
\end{equation}
namely, the matrix $M$ defined in \eqref{eq:M} is zero. Moreover, explicit Runge-Kutta methods are not explicit.
\end{thm}
For a symplectic Runge-Kutta method, we have
\begin{thm}\label{thm:symp2}
Let $(\mathbf{A, b, c})$ be an $s$-stage symplectic Runge-Kutta method. If $B(p), C(\eta)$ or $B(p), D(\eta)$ are satisfied with $p\leq 2\eta+1$, then the method is of order $p$.
\end{thm}

\begin{proof}
By using the $W$-transform \cite{Hairer81}, we know that the standard matrix $X_H$ of $\mathbf{A}$ is skew-symmetric, except the first entry $x_1=1/2$, and hence for the order conditions $C(\eta)$ and $D(\zeta)$, one always has $\eta=\zeta$. Next, we show that $C(\eta)$ implies $D(\eta)$ and vice versa. To that end, we note that $D(q)_i$ can be written as
\begin{equation}\label{eq:ss1}
\frac{1}{q}b_i-\sum_{j=1}^{s}b_jc_j^{q-1}a_{ji}=\frac{1}{q}b_ic_i^q.
\end{equation}
By virtue of the symplectic condition \eqref{eq:ss0}, one then has from \eqref{eq:ss1} that
$$\frac{1}{q}b_i-\sum_{j=1}^{s}c_j^{q-1}(b_ib_j-b_ia_{ij})=\frac{1}{q}b_ic_i^q,$$
which readily implies that
$$\sum_{j=1}^{s}c_j^{q-1}a_{ij}=\frac{1}{q}c_i^{q},$$
namely, $C(q)_i=0$. This completes the proof.
\end{proof}

\begin{prop}[\cite{Feng84}]\label{prop:ss1}
For a linear Hamiltonian system, the symmetry and symplecticity of a Runge-Kutta method are equivalent, and its stability function is of the form
\begin{equation}\label{eq:ss8}
R(z)=P(z)/P(-z), \quad z\in\mathbb{C}.
\end{equation}
\end{prop}

For a Hamiltonian system \eqref{eq:Ham1}, one can apply a pair of Runge-Kutta methods, $(\mathbf{A, b, c})$ and ${(\bar{\mathbf{A}}, \bar{\mathbf{b}}, \bar{\mathbf{c}}})$, respectively to the first and second equations to obtain a numerical integrator. This is called a partitioned Runge-Kutta method.  We have

\begin{thm}[\cite{Suris90},\cite{Serna92},\cite{Sun932}]\label{thm:ss5}
Consider a partitioned Runge-Kutta method $(\mathbf{A, b, c})-{(\bar{\mathbf{A}}, \bar{\mathbf{b}}, \bar{\mathbf{c}}})$. If their coefficients satisfy
\begin{equation}\label{eq:sprk1}
\begin{cases}
b_i\bar{a}_{ij}+\bar{b}_ja_{ji}-b_i\bar{b}_j=0,\quad i,j=1,...,s,\medskip\\
b_i=\bar{b_i},\quad i=1,...,s,
\end{cases}
\end{equation}
then the method is symplectic.
\end{thm}

It can be shown by straightforward calculations that for a linear Hamiltonian system, a partitioned Runge-Kutta method is symplectic if
\begin{equation}\label{eq:sss9}
R(z)\cdot \bar{R}(z)=P(z)/P(-z)\quad \mbox{and}\quad \bar{R}(z)=R(-z)^{-1},
\end{equation}
where $R(z)$ and $\bar{R}(z)$ are, respectively, the stability functions for $(\mathbf{A, b, c})$ and ${(\bar{\mathbf{A}}, \bar{\mathbf{b}}, \bar{\mathbf{c}}})$.
Henceforth, in order to simplify notations, we denote $\Phi:=(\mathbf{A, b, c})$ and $\bar{\Phi}:=(\bar{\mathbf{A}}, \bar{\mathbf{b}}, \bar{\mathbf{c}})$. We have

\begin{prop}\label{prop:ss9}
Consider a symplectic partitioned Runge-Kutta method $(\Phi, \bar{\Phi})$ and suppose that it is of order $p$. If $\Phi$ is of order $p$, then $\bar{\Phi}$ is at least of order $p$.
\end{prop}
\begin{proof}
Let $\bar{\Phi}$ be of order $q$. It can be readily seen that $(\Phi, \bar{\Phi})$ is of order $\min(p, q)$. Hence, $\bar{\Phi}$ is at least of order $p$.
\end{proof}

\section{Symmetric-adjoint and symplectic-adjoint methods}

\subsection{Symmetric-adjoint method}

Consider an $s$-stage Runge-Kutta method $(\mathbf{A, b, c})$ and its stability function given in \eqref{eq:sf2}. Set $\bar{\mathbf{A}}=\mathbf{e}\mathbf{b}^T-\mathbf{A}$,
\begin{equation}\label{eq:sa1}
\bar{\mathbf{A}}=(\mathbf{e}\mathbf{b}^T-\mathbf{A})=\left(
                 \begin{array}{cccc}
                   b_1 & b_2 & \cdots & b_s \\
                   \vdots  & \vdots & & \vdots \\
                   b_1 & b_2 & \cdots & b_s \\
                 \end{array}
               \right)-\left(
                         \begin{array}{cccc}
                           a_{11} &\cdots & a_{1s} \\
                           a_{21} & \cdots & a_{2s} \\
                           \vdots & & \vdots \\
                           a_{s1} & \cdots & a_{ss} \\
                         \end{array}
                       \right).
\end{equation}
One can see that there holds
\begin{equation}\label{eq:sa2}
\bar{a}_{ij}=b_j-a_{ij},\bar{c}_i=\sum_{j=1}^s\bar{a}_{ij}=\sum_{j=1}^s(b_j-a_{ij})=1-c_i,\quad \bar{b}_j=b_j,\quad i,j=1,\cdots,s.
\end{equation}
Therefore, one obtains a Runge-Kutta method $\bar{\Phi}$ that is of the same order of the original one $\Phi$. Suppose that the nodes of $\Phi$ satisfy $0\leq c_1\leq c_2\leq\cdots\leq c_s\leq 1$. By a permutation transform of the form,
$$
\tilde{P}\bar{\mathbf{A}}\tilde{P}^T=\tilde{P}(\mathbf{eb}^T-\mathbf{A})\tilde{P}^T:=\mathbf{A}^*,
$$
where
$$\tilde{P}=\left(
                 \begin{array}{ccccc}
                   0 & 0 & \cdots & 0 & 1 \\
                   \vdots  & & \ddots & \vdots \\
                   0 & 1 & \cdots & 0 & 0 \\
                   1 & 0 & \cdots & 0 & 0 \\
                 \end{array}
               \right),
$$
one can further obtain a Runge-Kutta method $\Phi^*=(\mathbf{A}^*, \mathbf{b}^*, \mathbf{c}^*)$ with
\begin{equation}
a_{ij}^*=b_{s+1-j}-a_{s+1-i,s+1-j},b_j^*=b_{s+1-j},c^*_i=1-c_{s+1-i},\quad i,j=1,\cdots,s.
\end{equation}
$\Phi^*$ is referred to as the {\it symmetric-adjoint method} of $\Phi$.

\begin{thm}[\cite{Hairer83,Hairer02}]\label{thm:sa1}
There hold the following results,
\begin{enumerate}
\item $(\Phi^*)^*=\Phi$;

\item If $\Phi^*=\Phi$, then $\Phi$ is a symmetric Runge-Kutta method, namely
$$
a_{ij}=b_{s+1-j}-a_{s+1-i,s+1-j},\quad b_j=b_{s+1-j},\quad c_i=1-c_{s+1-i},\quad i,j=1,\cdots,s.
$$

\item $\Phi$ and $\Phi^*$ possess the same order $p$, and $\Phi(h/2)\circ\Phi^*(h/2)$ is at least of order $p$.

\item $\Phi$ is symplectic if and only if $\Phi^*$ is symplectic. In other words, unless $\Phi^*$ is both symmetric and symplectic, one can always find a different symplectic method $\Phi^*$ from a given symplectic method $\Phi$.

\item The algebraic average of $\Phi$ and $\Phi^*$, namely $(\Phi+\Phi^*)/2$ is a symmetric Runge-Kutta method. In particular, if the quadrature nodes and weights, $\mathbf{c}$ and $\mathbf{b}$, are symmetric, then $\Phi^*=((\mathbf{A}+\mathbf{A}^*)/2, \mathbf{b}, \mathbf{c})$ possesses the same order of $\Phi$ (cf. \cite{Gan13}).

\end{enumerate}
\end{thm}

\subsection{Symplectic-adjoint method}

According to our earlier discussion, by the symplecticity condition for a partitioned Runge-Kutta method, namely \eqref{eq:sprk1}, one can solve it to obtain,
\begin{equation}\label{2.6}
\bar{a}_{ij}=b_j(1-\frac{a_{ji}}{b_i}),\quad b_i\neq0,\quad i,j=1,\cdots,s.
\end{equation}
Based on such an observation, we introduce the following notion of symplectic-adjoint methods.

\begin{defn}\label{def:smpa}
For an $s$-stage Runge-Kutta method $\Phi=(\mathbf{A, b, c})$, $b_i\neq 0, i=1, 2, \ldots, s$, we set $\Phi^{s*}=(\mathbf{A}^{s*}, \mathbf{b, c})$ with
\begin{equation}\label{eq:smpa1}
a_{ij}^{s*}=b_j(1-\frac{a_{ji}}{b_i}),\quad b_i\neq0,\quad i,j=1,\cdots,s.
\end{equation}
$\Phi^{s*}$ is called the symplectic-adjoint method of $\Phi$.
\end{defn}

\begin{thm}\label{thm:smpa1}
Let $\Phi$ be an $s$-stage Runge-Kutta method and $\Phi^{s*}$ be its symplectic-adjoint method. Then there holds the following properties
\begin{enumerate}
\item $(\Phi^{s*})^{s*}=\Phi$;

\item If $\Phi^{s*}=\Phi$, then $\Phi$ is symplectic;

\item $\Phi^{s*}$ and $\Phi$ possess the same order;

\item $\Phi^{s*}$ is symmetric if and only if $\Phi$ is symmetric;

\item The algebraic average of $\Phi$ and $\Phi^{s*}$, namely $((\mathbf{A}+\mathbf{A}^{s*})/2, \mathbf{b, c})$ is a symplectic method, and moreover, it possess the same order of $\Phi$.
\end{enumerate}
\end{thm}
\begin{proof}
The listed properties can be directly verified by using the definition of the symplectic-adjoint method.
\end{proof}

We note that the property $(5)$ in Theorem~\ref{thm:smpa1} provides a simple way of constructing symplectic Runge-Kutta method. For example, let us consider the Radau-IA method with $s=2$, one can proceed as follows,
\begin{equation*}
  \Phi:
  \begin{tabular}{c|cc}
  $0$  & $\frac{1}{4}$ & $-\frac{1}{4}$ \\
  $\frac{2}{3}$ & $\frac{1}{4}$  & $\frac{5}{12}$ \\
  \hline
    & $\frac{1}{4}$ & $\frac{3}{4}$ \\
\end{tabular}
\rightarrow
\Phi^{s*}:
  \begin{tabular}{c|cc}
  $0$  & $0$ & $0$ \\
  $\frac{2}{3}$ & $\frac{1}{3}$  & $\frac{1}{3}$ \\
  \hline
    & $\frac{1}{4}$ & $\frac{3}{4}$ \\
\end{tabular}
\Rightarrow
  \mbox{ by using (5) in Theorem~\ref{thm:smpa1}}
\Rightarrow
  \begin{tabular}{c|cc}
  $0$  & $\frac{1}{8}$ & $-\frac{1}{8}$ \\
  $\frac{2}{3}$ & $\frac{7}{24}$  & $\frac{3}{8}$ \\
  \hline
    & $\frac{1}{4}$ & $\frac{3}{4}$ \\
  \end{tabular},
\end{equation*}
where the resulting method is symplectic and known as the Radau-IB method with $s=2$.

Finally, we present an interesting property of the symmetric-adjoint and symplectic-adjoint methods.

\begin{thm}\label{thm:final}
Let $\Phi$ be an $s$-stage, $p$-order Runge-Kutta method which is assumed to be not $A$-stable. Then $\Phi^*$ and $\Phi^{s*}$ are both $A$-stable and at least of order $p$. Further more, if the stability function $R(z)$ of $\Phi$ satisfies $\lim_{|z|\rightarrow+\infty} R(z)=\infty$, then $\Phi^*$ and $\Phi^{s*}$ are $L$-stable.
\end{thm}

\begin{proof}
By virtue of property (3) in Theorem~\ref{thm:sa1}, one sees that $\Phi^*$ is of order $p$. Since $\Phi^*(h/2)\circ\Phi(h/2)$ is symmetric, we know that $\Phi^*$ is $A$-stable. If the stability function $R(z)$ of $\Phi$ satisfies $\lim_{|z|\rightarrow+\infty} R(z)=\infty$, by the definition of $\Phi^*$, one readily verifies that $\Phi^*$ is $L$-stable. As for $\Phi^{s*}$, one can also obtain the statement of the theorem by using property (3) in Theorem~\ref{thm:smpa1}, and the fact that $R(z)R^{s*}(z)=P(z)/P(-z)$ from \eqref{eq:sss9}.

The proof is complete.
\end{proof}

To illustrate some interesting applications of Theorem~\ref{thm:final}, we next consider some specific Runge-Kutta methods, ranging from order 1 to order 4. The original methods are all explicit, but their adjoint methods are all $L$-stable.
\begin{equation*}
  \Phi:
  \begin{tabular}{c|c}
  $0$  & $0$ \\
  \hline
    & $1$ \\
\end{tabular},\quad
R(z)=(1+z),\quad
\Phi^{s*}:
  \begin{tabular}{c|c}
  $1$  & $1$ \\
  \hline
    & $1$ \\
\end{tabular},\quad
R^{s*}(z)=(1-z)^{-1},
\mbox{ $L$-stable };
\end{equation*}

\begin{equation*}
  \Phi:
  \begin{tabular}{c|cc}
  $0$  & $0$ & $0$ \\
  $1$ & $1$  & $0$ \\
  \hline
    & $\frac{1}{2}$ & $\frac{1}{2}$ \\
\end{tabular},\quad
R(z)=(1+z+\frac{1}{2}z^2),\quad
\Phi^{s*}:
  \begin{tabular}{c|cc}
  $0$  & $\frac{1}{2}$ & $-\frac{1}{2}$ \\
  $1$ & $\frac{1}{2}$  & $\frac{1}{2}$ \\
  \hline
    & $\frac{1}{2}$ & $\frac{1}{2}$ \\
\end{tabular},\quad
R^{s*}(z)=R(-z)^{-1},
\mbox{ $L$-stable};
\end{equation*}

\begin{equation*}
  \Phi:
  \begin{tabular}{c|ccc}
  $0$  & $0$ & $0$ & $0$ \\
  $\frac{1}{2}$ & $\frac{1}{2}$  & $0$ & $0$ \\
  $1$  & $-1$ & $2$ & $0$ \\
  \hline
    & $\frac{1}{6}$ & $\frac{2}{3}$ & $\frac{1}{6}$ \\
\end{tabular},\quad
R(z)=(1+z+\frac{1}{2}z^2+\frac{1}{6}z^3),\quad
\Phi^{s*}:
  \begin{tabular}{c|ccc}
  $0$  & $\frac{1}{6}$ & $-\frac{4}{3}$ & $\frac{7}{6}$ \\
  $\frac{1}{2}$ & $\frac{1}{6}$  & $\frac{2}{3}$ & $-\frac{1}{3}$ \\
  $1$  & $\frac{1}{6}$ & $\frac{2}{3}$ & $\frac{1}{6}$ \\
  \hline
    & $\frac{1}{6}$ & $\frac{2}{3}$ & $\frac{1}{6}$ \\
\end{tabular},
\end{equation*}
where $R^{s*}(z)=R(-z)^{-1}$ and is $L$-stable, and
\begin{equation*}
  \Phi:
  \begin{tabular}{c|cccc}
  $0$  & $0$ & $0$ & $0$ & $0$ \\
  $\frac{1}{2}$ & $\frac{1}{2}$  & $0$ & $0$ & $0$ \\
  $\frac{1}{2}$  & $0$ & $\frac{1}{2}$ & $0$ & $0$ \\
  $1$  & $0$ & $0$ & $1$ & $0$ \\
  \hline
    & $\frac{1}{6}$ & $\frac{1}{3}$ & $\frac{1}{3}$ & $\frac{1}{6}$ \\
\end{tabular},\quad
R(z)=\sum_{i=1}^4(\frac{1}{i!}z^i+1),\quad
\Phi^{s*}:
  \begin{tabular}{c|cccc}
  $0$  & $\frac{1}{6}$ & $-\frac{2}{3}$ & $\frac{1}{3}$ & $\frac{1}{6}$ \\
  $\frac{1}{2}$ & $\frac{1}{6}$  & $\frac{1}{3}$ & $-\frac{1}{6}$ & $\frac{1}{6}$ \\
  $\frac{1}{2}$  & $\frac{1}{6}$ & $\frac{1}{3}$ & $\frac{1}{3}$ & $-\frac{1}{3}$ \\
  $1$  & $\frac{1}{6}$ & $\frac{1}{3}$ & $\frac{1}{3}$ & $\frac{1}{6}$ \\
  \hline
    & $\frac{1}{6}$ & $\frac{1}{3}$ & $\frac{1}{3}$ & $\frac{1}{6}$ \\
\end{tabular},
\end{equation*}
where $R^{s*}(z)=R(-z)^{-1}$ and is $L$-stable. For the classical Lobatto-type and Radau-type methods (cf. \cite{Hairer83}), we have the following interesting connections:

\begin{enumerate}

\item If $\Phi$ is a Lobatto-III-A method, then $\Phi^{s*}$ is a Lobatto-III-B method. Furthermore, by using property (5) in Theorem~\ref{thm:smpa1}, one can obtain Lobatto-III-S methods.

\item If $\Phi$ is a Lobatto-III-C method, then $\Phi^{s*}$ is a Lobatto-III-E method (cf. \cite{Norsett81}); e.g.,
\begin{equation*}
  \Phi:
  \begin{tabular}{c|ccc}
  $0$  & $\frac{1}{6}$ & $-\frac{1}{3}$ & $\frac{1}{6}$ \\
  $\frac{1}{2}$ & $\frac{1}{6}$  & $\frac{5}{12}$ & $-\frac{1}{12}$ \\
  $1$  & $\frac{1}{6}$ & $\frac{2}{3}$ & $\frac{1}{6}$ \\
  \hline
    & $\frac{1}{6}$ & $\frac{2}{3}$ & $\frac{1}{6}$ \\
\end{tabular}\quad
\Rightarrow\quad
\Phi^{s*}:
  \begin{tabular}{c|ccc}
  $0$  & $0$ & $0$ & $0$ \\
  $\frac{1}{2}$ & $\frac{1}{4}$  & $\frac{1}{4}$ & $0$ \\
  $1$  & $0$ & $1$ & $0$ \\
  \hline
    & $\frac{1}{6}$ & $\frac{2}{3}$ & $\frac{1}{6}$ \\
\end{tabular}\quad
\Rightarrow\quad
\Phi^{s}:
  \begin{tabular}{c|ccc}
  $0$  & $\frac{1}{12}$ & $-\frac{1}{6}$ & $\frac{1}{12}$ \\
  $\frac{1}{2}$ & $\frac{5}{24}$  & $\frac{1}{3}$ & $-\frac{1}{24}$ \\
  $1$  & $\frac{1}{12}$ & $\frac{5}{6}$ & $\frac{1}{12}$ \\
  \hline
    & $\frac{1}{6}$ & $\frac{2}{3}$ & $\frac{1}{6}$ \\
\end{tabular}.
\end{equation*}

\item For Radau-type methods, we have the following examples
\begin{equation*}
  \Phi:
  \mbox{Radau-I-A, $s=2$, both algebraically and $A$-stable},
  \begin{tabular}{c|cc}
  $0$  & $\frac{1}{4}$ & $-\frac{1}{4}$ \\
  $\frac{2}{3}$ & $\frac{1}{4}$  & $\frac{5}{12}$ \\
  \hline
    & $\frac{1}{4}$ & $\frac{3}{4}$ \\
\end{tabular}
\end{equation*}
\begin{equation*}
\Rightarrow\quad
\Phi^{s*}:
  \begin{tabular}{c|cc}
  $0$  & $0$ & $0$ \\
  $\frac{2}{3}$ & $\frac{1}{3}$  & $\frac{1}{3}$ \\
  \hline
    & $\frac{1}{4}$ & $\frac{3}{4}$
\end{tabular},\ \
  \mbox{ non-$A$-stable }
\end{equation*}
\begin{equation*}
\Rightarrow\quad
  \Phi^s:
  \mbox{ Radau-I-B, symplectic}, \quad
  \begin{tabular}{c|cc}
  $0$  & $\frac{1}{8}$ & $-\frac{1}{8}$ \\
  $\frac{2}{3}$ & $\frac{7}{24}$  & $\frac{3}{8}$ \\
  \hline
    & $\frac{1}{4}$ & $\frac{3}{4}$ \\
\end{tabular}
\end{equation*}
\begin{equation*}
\Rightarrow\quad
  (\Phi^s)^*:
  \mbox{ Radau-II-B, symplectic },\quad
  \begin{tabular}{c|cc}
  $\frac{1}{3}$  &$\frac{3}{8}$ & $-\frac{1}{24}$ \\
  $1$ & $\frac{7}{8}$ & $\frac{1}{8}$ \\
  \hline
    & $\frac{3}{4}$ & $\frac{1}{4}$ \\
  \end{tabular}\ \ .
\end{equation*}
On the other hand,
\begin{equation*}
  \Phi:
  \mbox{Radau-I-A, both algebraically and $L$-stable, }
  \begin{tabular}{c|cc}
  $0$  & $\frac{1}{4}$ & $-\frac{1}{4}$ \\
  $\frac{2}{3}$ & $\frac{1}{4}$  & $\frac{5}{12}$ \\
  \hline
    & $\frac{1}{4}$ & $\frac{3}{4}$ \\
\end{tabular}
\end{equation*}
\begin{equation*}
\Rightarrow\quad
\Phi^{s*}:
  \begin{tabular}{c|cc}
  $0$  & $0$ & $0$ \\
  $\frac{2}{3}$ & $\frac{1}{3}$  & $\frac{1}{3}$ \\
  \hline
    & $\frac{1}{4}$ & $\frac{3}{4}$
\end{tabular}\ \ ,
  \mbox{ non-$A$-stable }
\end{equation*}
\begin{equation*}
\Rightarrow\quad
  (\Phi^{s*})^*:
  \mbox{ Radau-II-A, both algebraically and $L$-stable },\quad
  \begin{tabular}{c|cc}
  $\frac{1}{3}$  &$\frac{5}{12}$ & $-\frac{1}{12}$ \\
  $1$ & $\frac{3}{4}$ & $\frac{1}{4}$ \\
  \hline
    & $\frac{3}{4}$ & $\frac{1}{4}$ \\
  \end{tabular}.
\end{equation*}
In principle, by following a similar manner, from Radau-I-A type methods, one can obtain all the Radau type methods through the symmetric-adjoint and symplectic-adjoint approaches.

\end{enumerate}

\section{Construction of high-order explicit Runge-Kutta methods}
\subsection{Construction of explicit Runge-Kutta methods with stage 6 and order 5}

In this section, we consider the application of symmetric-adjoint and symplectic-adjoint methods in significantly simplifying the construction of high-order explicit Runge-Kutta methods. To that end, we present the construction procedure of a class of novel explicit Runge-Kutta methods of order 5 and stage 6.

Let $\Phi=(\mathbf{A, b, c})$ is the method to be determined, of the following form,
\begin{equation}\label{method2.1}
\Phi:
\begin{tabular}{c|cccccc}
$c_1=0$ & $0$\\
$c_2$ & $a_{21}$ & $0$\\
$c_3$ & $a_{31}$ & $a_{32}$ & $0$\\
$c_4$ & $a_{41}$ & $a_{42}$ & $a_{43}$ & $0$\\
$c_5$ & $a_{51}$ & $a_{52}$ & $a_{53}$ & $a_{54}$ & $0$\\
$c_6$ & $a_{61}$ & $a_{62}$ & $a_{63}$ & $a_{64}$ & $a_{65}$ & $0$\\
\hline
$ $ & $b_1$ & $b_2$ & $b_3$ & $b_4$ & $b_5$ & $b_6$
\end{tabular}
\end{equation}
where it is required that $b_i\neq 0$, $i=1,2,\ldots, 6$ and there are totally 21 coefficients to be found. According to the order conditions (cf. \cite{Hairer83}), there are the following order conditions satisfied by the aforesaid coefficients,
\begin{equation*}
ERK(5):
\begin{cases}
t_{41}: \sum b_i=1,\\
t_{42}: \sum b_ic_i=1/2,\\
t_{43}: \sum b_ic_i^2=1/3,\\
t_{44}: \sum b_ic_i^3=1/4,\\
t_{45}: \sum b_ia_{ij}c_j=1/6,\\
t_{46}: \sum b_ic_ia_{ij}c_j=1/8,\\
t_{47}: \sum b_ia_{ij}c_j^2=1/12,\\
t_{48}: \sum b_ia_{ij}a_{jk}c_k=1/24,
\end{cases}
\quad
\begin{cases}
t_{59}:  \sum b_ic_i^4=1/5,\\
t_{510}: \sum b_ic_i^2a_{ij}c_j=1/10,\\
t_{511}: \sum b_ic_ia_{ij}c_j^2=1/15,\\
t_{512}: \sum b_ic_ia_{ij}a_{jk}c_k=1/30,\\
t_{513}: \sum b_i(\sum a_{ij}c_j)^2=1/20,\\
t_{514}: \sum b_ia_{ij}c_j^3=1/20,\\
t_{515}: \sum b_ia_{ij}c_ja_{jk}c_k=1/40,\\
t_{516}: \sum b_ia_{ij}a_{jk}c_k^2=1/60,\\
t_{517}: \sum b_ia_{ij}a_{jk}a_{km}c_m=1/120.
\end{cases}
\end{equation*}
By Definition~\ref{def:smpa}, one has
\begin{equation}\label{method2.2}
\Phi^{s*}:
\begin{tabular}{c|c}
$\mathbf{c}$  & $\mathbf{A}^{s*}$  \\
  \hline
    $ $  & $\mathbf{b}^{T}$ \\
\end{tabular}
\mbox{\ \ where $a_{ij}^{s*}=b_j(1-\frac{a_{ji}}{b_i})$, $b_i\neq0$, i, j=1,...,6,}
\end{equation}
\begin{equation*}
\begin{tabular}{c|cccccc}
$c_1=0$ & $b_1$ & $b_2(1-\frac{a_{21}}{b_1})$ & $b_3(1-\frac{a_{31}}{b_1})$ & $b_4(1-\frac{a_{41}}{b_1})$ & $b_5(1-\frac{a_{51}}{b_1})$ & $b_6(1-\frac{a_{61}}{b_1})$\\
$c_2$ & $b_1$ & $b_2$ & $b_3(1-\frac{a_{32}}{b_2})$ & $b_4(1-\frac{a_{42}}{b_2})$ & $b_5(1-\frac{a_{52}}{b_2})$ & $b_6(1-\frac{a_{62}}{b_2})$\\
$c_3$ & $b_1$ & $b_2$ & $b_3$ & $b_4(1-\frac{a_{43}}{b_3})$ & $b_5(1-\frac{a_{53}}{b_3})$ & $b_6(1-\frac{a_{63}}{b_3})$\\
$c_4$ & $b_1$ & $b_2$ & $b_3$ & $b_4$ & $b_5(1-\frac{a_{54}}{b_4})$ & $b_6(1-\frac{a_{64}}{b_4})$\\
$c_5$ & $b_1$ & $b_2$ & $b_3$ & $b_4$ & $b_5$ & $b_6(1-\frac{a_{65}}{b_5})$\\
$c_5$ & $b_1$ & $b_2$ & $b_3$ & $b_4$ & $b_5$ & $b_6$\\
\hline
$ $   & $b_1$ & $b_2$ & $b_3$ & $b_4$ & $b_5$ & $b_6$
\end{tabular}\ \ .
\end{equation*}
By property (5) in Theorem~\ref{thm:smpa1}, we let $\Phi^s=((\mathbf{A}+\mathbf{A}^{s*})/2, \mathbf{b}, \mathbf{c})$, given by
\begin{equation}\label{2.3}
\begin{tabular}{c|cccccc}
$c_1=0$ & $b_1/2$ & $b_2(1-\frac{a_{21}}{b_1})/2$ & $b_3(1-\frac{a_{31}}{b_1})/2$ & $b_4(1-\frac{a_{41}}{b_1})/2$ & $b_5(1-\frac{a_{51}}{b_1})/2$ & $b_6(1-\frac{a_{61}}{b_1})/2$\\
$c_2$ & $(b_1+a_{21})/2$ & $b_2/2$ & $b_3(1-\frac{a_{32}}{b_2})/2$ & $b_4(1-\frac{a_{42}}{b_2})/2$ & $b_5(1-\frac{a_{52}}{b_2})/2$ & $b_6(1-\frac{a_{62}}{b_2})/2$\\
$c_3$ & $(b_1+a_{31})/2$ & $(b_2+a_{32})/2$ & $b_3/2$ & $b_4(1-\frac{a_{43}}{b_3})/2$ & $b_5(1-\frac{a_{53}}{b_3})/2$ & $b_6(1-\frac{a_{63}}{b_3})/2$\\
$c_4$ & $(b_1+a_{41})/2$ & $(b_2+a_{42})/2$ & $(b_3+a_{43})/2$ & $b_4/2$ & $b_5(1-\frac{a_{54}}{b_4})/2$ & $b_6(1-\frac{a_{64}}{b_4})/2$\\
$c_5$ & $(b_1+a_{51})/2$ & $(b_2+a_{52})/2$ & $(b_3+a_{53})/2$ & $(b_4+a_{54})/2$ & $b_5/2$ & $b_6(1-\frac{a_{65}}{b_5})/2$\\
$c_5$ & $(b_1+a_{61})/2$ & $(b_2+a_{62})/2$ & $(b_3+a_{63})/2$ & $(b_4+a_{64})/2$ & $(b_5+a_{65})/2$ & $b_6/2$\\
\hline
$ $   & $b_1$ & $b_2$ & $b_3$ & $b_4$ & $b_5$ & $b_6$
\end{tabular}.
\end{equation}
We know that $\Phi^s$ is a symplectic Runge-Kutta method of order $5$ and stage 6. It is further assumed that $\Phi^{s*}=\Phi^{*}$  and hence $\Phi^s$ becomes a symmetric and symplectic method. By our earlier study, if $\Phi$ satisfies the simplified order conditions $B(5)$, $C(\eta)$ and $D(\zeta)$ with $\eta=\zeta=2$, then it is of order 5.

By $\Phi^{s*}=\Phi^{*}$, we first have
\begin{equation}\label{2.4}
b_ia_{ij}=b_ja_{s+1-j,s+1-i}, i=2,...,6, j=1,...,6,
\end{equation}
namely,
\begin{equation*}
b_6a_{65}=b_5a_{21}, b_6a_{64}=b_4a_{31}, b_6a_{63}=b_3a_{41}, b_6a_{62}=b_2a_{51}, b_5a_{54}=b_4a_{32}, b_5a_{53}=b_3a_{42}.
\end{equation*}
that is,
\begin{equation*}
b_1a_{65}=b_2a_{21}=b_2c_2, b_1a_{64}=b_3a_{31}, b_1a_{63}=b_3a_{41}, b_1a_{62}=b_2a_{51}, b_2a_{54}=b_3a_{32}, b_2a_{53}=b_3a_{42}.
\end{equation*}
By $C(1)$ and $D(1)$ for $\Phi$, we have
\begin{equation}\label{2.5}
C(1): \sum_{j=1}a_{ij}=c_i, i=1,...,6, \mbox{i.e.},\ \
\begin{cases}
a_{21}=c_2,\\
a_{31}+a_{32}=c_3,\\
\vdots\\
a_{61}+a_{62}+...+a_{65}=1.
\end{cases}
\end{equation}

\begin{equation}\label{eq:2.6}
D(1): \sum_{i=1}b_ia_{ij}=b_j(1-c_j), j=1,...,6, \mbox{i.e., }
\begin{cases}
b_2a_{21}+b_3a_{31}+...+b_6a_{61}=b_1(1-c_1)=b_1,\\
b_3a_{32}+b_4a_{42}+b_5a_{52}+b_6a_{62}=b_2(1-c_2),\\
b_4a_{43}+b_5a_{53}+b_6a_{63}=b_3(1-c_3),\\
b_5a_{54}+b_6a_{64}=b_4(1-c_4),\\
b_6a_{65}=b_5(1-c_5),\\
0=b_6(1-c_6)\Rightarrow c_6=1.
\end{cases}
\end{equation}
From \eqref{eq:2.6}, one easily obtains that

\begin{equation}\label{eq:2.6b}
b_3A_3+b_4A_4+b_5A_5+b_6A_6=\frac{1}{6},
\end{equation}
where
\[
A_3=a_{32}c_2, A_4=a_{42}c_2+a_{43}c_3, A_5=a_{52}c_2+a_{53}c_3+a_{54}c_4, A_6=a_{62}c_2+a_{63}c_3+a_{64}c_4+a_{65}c_5.
\]
By Proposition~\ref{prop:2}, we see that the requirement of the first equation in \eqref{eq:2.6} can be removed. In addition, it can directly verified that if $\Phi$ satisfies $C(1)$ and $D(1)$, then $\Phi^s$ satisfies the corresponding order conditions $C(1)$ and $D(1)$. Hence, in order to imposed those order conditions on $\Phi^s$, it suffices for us to consider \eqref{2.5} and \eqref{eq:2.6}.

By $\Phi^{s*}=\Phi^s$, we have $c_1=0, c_2=1-c_5, c_3=1-c_4, c_6=1-c_1$ and $b_i=b_{s+1-i}, i=1,...,6,$ which together with \eqref{2.4} can simplify \eqref{eq:2.6} to
\begin{equation}\label{eq:2.6n}
\begin{cases}
b_6a_{62}+b_5a_{52}+b_4a_{42}+b_3a_{32}=b_2(1-c_2),\\
b_6a_{63}+b_5a_{53}+b_4a_{43}=b_3(1-c_3), b_5a_{53}=b_3c_{42},\\
b_6a_{64}+b_5a_{54}=b_4(1-c_4)=b_3c_3, b_5a_{54}=b_3a_{32},\\
b_6a_{65}=b_5(1-c_5)=b_2c_2.
\end{cases}
\end{equation}
Now, we treat $c_1, c_2$ as parameters with freedom, and it is clear that if $b_1, b_2, b_3$ and $a_{32}, a_{42}, a_{43}, a_{52}$ are determined, then so are the rest of the coefficients by \eqref{2.4}--\eqref{eq:2.6}.

We proceed to determine the aforesaid parameters. By Proposition~\ref{prop:2} and Theorem~\ref{thm:symp2}, we could further simply the relevant order conditions. First, using $B(5)$, we have
\begin{equation}\label{2.7}
\begin{cases}
b_1+b_2+b_3=1/2,\\
b_2c_2(1-c_2)+b_3c_3(1-c_3)=1/12,\\
b_2c_2^2(1-c_2)^2+b_3c_3^2(1-c_3)^2=1/60,
\end{cases}
\end{equation}
which further yields
\begin{equation}\label{2.8}
\begin{cases}
\displaystyle{b_1=b_6=\frac{1}{2}-\frac{5(c_2(1-c_2)+c_3(1-c_3))-1}{60c_2c_3(1-c_2)(1-c_3)}},\medskip\\
\displaystyle{b_2=b_5=\frac{-\frac{1}{12}c_3(1-c_3)+\frac{1}{60}}{c_2(1-c_2)(c_2(1-c_2)-c_3(1-c_3))}},\medskip\\
\displaystyle{b_3=b_4=\frac{\frac{1}{12}c_2(1-c_2)-\frac{1}{60}}{c_3(1-c_3)(c_2(1-c_2)-c_3(1-c_3))}}.
\end{cases}
\end{equation}
Next, by $D(2):\sum_{i=1}^6b_ic_ia_{ij}^s=\frac{1}{2}b_j(1-c_j^2), j=1,...,6$, where $a_{ij}^s=\frac{1}{2}(a_{ij}+a_{ij}^{s*}), a_{ij}^{s*}=b_j(1-\frac{a_{ji}}{b_i}), i,j=1,...,6$, we have by direct calculations that $D(2):\sum_{i=1}^6(b_ia_{ij}-b_ja_{ji})c_i+\sum_{i=1}^6b_ic_ib_j=b_j(1-c_j^2), j=1,...,6$, where $b_j\sum_{i=1}^6b_ic_i=\frac{1}{2}b_j$. Note that $D(1):\sum_{i=1}^6b_ia_{ij}=b_j(1-c_j), j=1,...6$. Subtracting $D(1)$ from $D(2)$, we finally have
\begin{equation}\label{2.9}
\sum_{i=1}^6((1-c_i)b_ia_{ij}+c_ib_ja_{ji})=b_j(\frac{1}{2}-c_j+c_j^2):=D_j,  j=1,...,6,
\end{equation}
namely,
\begin{equation}\label{eq:2.9b}
\begin{cases}
b_2(1-c_2)a_{21}+b_3(1-c_3)a_{31}+b_4(1-c_4)a_{41}+b_5(1-c_5)a_{51}=b_1(\frac{1}{2}-c_1+c_1^2)=\frac{1}{2}b_1=D_1,\\
b_3(1-c_3)a_{32}+b_4(1-c_4)a_{42}+b_5(1-c_5)a_{52}=b_2(\frac{1}{2}-c_2+c_2^2)=D_2,\\
b_3A_3+b_4(1-c_4)a_{43}+b_5(1-c_5)a_{53}=b_3(\frac{1}{2}-c_3+c_3^2)=D_3,\\
b_4A_4+b_5(1-c_5)a_{54}=b_4(\frac{1}{2}-c_4+c_4^2)=D_4,\\
b_5A_5=b_5(\frac{1}{2}-c_5+c_5^2)=D_5,\\
\mbox{where}\ \ A_3=a_{32}c_2, A_4=(a_{42}c_2+a_{43}c_3), A_5=(a_{52}c_2+a_{53}c_3+a_{54}c_4),\\
b_6A_6=b_6(\frac{1}{2}-c_6+c_6^2):=D_6,\ A_6=(a_{62}c_2+a_{63}c_3+a_{64}c_4+a_{65}c_5),
\end{cases}
\end{equation}
and also $b_5(1-c_5)a_{53}=b_3a_{42}c_2,b_5(1-c_5)a_{54}=b_3A_3$.

By Propositions \ref{prop:1}-\ref{prop:2}, we see that there are only two independent relationships in \eqref{eq:2.9b}, which can be chosen be the second the third ones, namely,
\begin{equation}\label{eq:c1}
\begin{cases}
b_3(1-c_3)a_{32}+b_4(1-c_4)a_{42}+b_5(1-c_5)a_{52}=D_2,\\
b_3A_3+b_4A_4=D_4=D_3,
\end{cases}
\end{equation}
where we use $b_ia_{ij}=b_ja_{s+1-j,s+1-i}$ in the second equality of \eqref{eq:c1}.

By combining the order conditions of \eqref{2.3} and \eqref{eq:c1}, we have
\begin{equation}\label{2.10}
SSRK(5):
\begin{cases}
B(5):\eqref{2.7},\\
\Phi^{s*}=\Phi^*:\eqref{2.4},\\
D(1):\eqref{2.6}, C(1):\eqref{2.5},\\
D(2):\eqref{2.9}.
\end{cases}
\end{equation}
By \eqref{2.10}, as soon as $a_{32}, a_{42}, a_{43}, a_{52}$ are determined, then the rest of the coefficients of $\Phi$ can be determined accordingly.

Since $R(z)$ is a 5th order approximation to $e^z$, namely,
\begin{equation}\label{2.11}
R(z)=\sum_{i=1}^5\frac{1}{i}z^i+\frac{\alpha}{6!}z^6,\ \ z=\lambda h,
\end{equation}
then by comparing the coefficients, we have
\begin{equation*}
\begin{cases}
\displaystyle{t_{48}: \sum b_ja_{jk}a_{kl}c_l=\frac{1}{24}\Rightarrow b_3(1-c_3)A_3+b_4(1-c_4)A_4=\frac{1}{24}-D_5(1-c_5),}\\
\displaystyle{t_{517}: \sum b_ja_{jk}a_{kl}a_{lm}c_m=\frac{1}{120}\Rightarrow(D_3-b_3A_3)A_3+(D_4-b_4A_4)A_4=\frac{1}{120},}\\
\displaystyle{t_{620}: \sum b_ja_{jk}a_{kl}a_{lm}a_{mp}c_p=\frac{\alpha}{720}\Rightarrow (D_4-b_4A_4)a_{43}A_3=\frac{\alpha}{720}.}
\end{cases}
\end{equation*}
By $\Phi^{s*}=\Phi^*$, we can further derive that
\begin{equation}\label{2.11}
\begin{cases}
\displaystyle{b_3(1-c_3)A_3+b_4(1-c_4)A_4=\frac{1}{24}-D_2c_2,}\\
\displaystyle{b_3A_3A_4=\frac{1}{240},}\medskip\\
\displaystyle{b_3a_{43}A_3^2=\frac{\alpha}{720}.}
\end{cases}
\end{equation}
Finally, we can verify by straightforward calculations that if $SSRK(5)$ and \eqref{2.11} are fulfilled, then the order conditions listed in $ERK(5)$ are all fulfilled. That is, the Runge-Kutta method $\Phi$ is of order 5.

By our derivation so far, it suffices for us to consider the following system of linear equations,
\begin{equation}\label{2.13}
\begin{cases}
\displaystyle{b_3(1-c_3)a_{32}+b_4(1-c_4)a_{42}+b_5(1-c_5)a_{52}=D_2},\\
\displaystyle{b_3A_3+b_4A_4=D_4=D_3,}\\
\displaystyle{b_3(1-c_3)A_3+b_4(1-c_4)A_4=\frac{1}{24}-D_2c_2,}\\
\displaystyle{2b_3A_3A_4=\frac{1}{120}.}
\end{cases}
\end{equation}
Set $X=b_4A_4, Y=b_3A_3$, by the 2nd and 3rd relations in \eqref{2.13}, we can obtain
\begin{equation}\label{2.14}
\begin{cases}
\displaystyle{X=\frac{D_3(1-c_3)-(\frac{1}{24}-D_2c_2)}{1-2c_3}},\\
\displaystyle{Y=\frac{\frac{1}{24}-(D_2c_2+D_3c_3)}{1-2c_3}.}
\end{cases}
\end{equation}
From $\frac{1}{24}=\frac{1}{2}\cdot\frac{1}{12}=\frac{1}{2}(b_2c_2(1-c_2)+b_3c_3(1-c_3))$, we can further deduce that
\begin{equation}\label{2.15}
\begin{cases}
\displaystyle{X=\frac{(\frac{1}{2}-c_3)(1-c_3)^2b_3-(\frac{1}{2}-c_2)c_2^2b_2}{1-2c_3},}\\
\displaystyle{Y=\frac{(\frac{1}{2}-c_2)c_2^2b_2+(\frac{1}{2}-c_3)c_3^2b_3}{1-2c_3}.}
\end{cases}
\end{equation}
Substituting \eqref{2.15} into the 4th equation in \eqref{2.13}, namely $XY=b_3/240$, one then has
\begin{align*}
& \left\{\left(\frac{1}{2}-c_3\right)(1-c_3)^2b_3-\left(\frac{1}{2}-c_2\right)c_2^2b_2\right\}\cdot\left\{\left(\frac{1}{2}-c_2\right)c_2^2b_2+\left(\frac{1}{2}-c_3\right)c_3^2b_3\right\}\\
=& \frac{1}{240}b_3(1-2c_3)^2.
\end{align*}
From 
\[
\frac{1}{240}=\frac{1}{4}\cdot\frac{1}{60}=\frac{1}{4}(b_2c_2^2(1-c_2)^2+b_3c_3^2(1-c_3)^2),
\] 
we can further deduce that
\[
\begin{split}
&\left\{\left(\frac{1}{2}-c_3\right)(1-c_3)^2-\left(\frac{1}{2}-c_2\right)c_2^2\frac{b_2}{b_3}\right\}\cdot\left\{\left(\frac{1}{2}-c_2\right)c_2^2\frac{b_2}{b_3}+\left(\frac{1}{2}-c_3\right)c_3^2\right\}\\
=&\left(\frac{1}{2}-c_3\right)\left\{c_2^2(1-c_2)^2\frac{b_2}{b_3}+c_3^2(1-c_3)^2\right\}.
\end{split}
\]
By comparing the coefficients in the above equation, one has,
\begin{equation*}
\frac{b_2}{b_3}c_2^4\left\{\left(\frac{1}{2}-c_3\right)^2+\frac{b_2}{b_3}\left(\frac{1}{2}-c_2\right)^2\right\}=0,\ b_i\neq0, i=1,...,6,\ c_2\neq 0,
\end{equation*}
which gives 
\begin{equation}\label{2.16}
b_2\left(\frac{1}{2}-c_2\right)^2+b_3\left(\frac{1}{2}-c_3\right)^2=0.
\end{equation}
Substituting $b_2$ and $b_3$ in \eqref{2.8} into \eqref{2.16}, we have
\begin{equation*}
(\frac{1}{2}-c_2)^2\frac{\frac{1}{12}c_3(1-c_3)+\frac{1}{60}}{c_2(1-c_2)}+(\frac{1}{2}-c_3)^2\frac{\frac{1}{12}c_2(1-c_2)-\frac{1}{60}}{c_3(1-c_3)}=0,
\end{equation*}
which  yield
\begin{equation}\label{2.17}
(1-2c_2)^2c_3(1-c_3)\left(-c_3(1-c_3)+\frac{1}{5}\right)+(1-2c_3)^2c_2(1-c_2)\left(c_2(1-c_2)-\frac{1}{5}\right)=0.
\end{equation}
It follows by a straightforward calculation from \eqref{2.17} that
\begin{equation}\label{2.18}
c_2(1-c_2)+c_3(1-c_3)-4c_2(1-c_2)c_3(1-c_3)-\frac{1}{5}=0.
\end{equation}
Plugging \eqref{2.18} into \eqref{2.8}, one obtains that $b_1=b_6=\frac{1}{6}$. Next, \eqref{2.18} can be reformulated as
\begin{equation*}
c_3^2-c_3+\frac{\frac{1}{5}-c_2(1-c_2)}{1-4c_2(1-c_2)}=0,
\end{equation*}
which readily gives
\begin{equation}\label{2.19}
c_3=\frac{1}{2}-\frac{\sqrt{5}}{10(1-2c_2)},\ 0<c_2<\frac{1}{2}.
\end{equation}

Finally, by $b_1=b_6=\frac{1}{6}$ and \eqref{2.8}, we have
\begin{equation}\label{2.20}
\begin{cases}
\displaystyle{b_1=b_6=\frac{1}{6},}\\
\displaystyle{b_2=b_5=\frac{-\frac{1}{12}(1-2c_3)^2}{c_3(1-c_3)-c_2(1-c_2)}},\medskip\\
\displaystyle{b_3=b_4=\frac{\frac{1}{12}(1-2c_2)^2}{c_3(1-c_3)-c_2(1-c_2)}\ \ \mbox{where}\ c_3=\frac{1}{2}-\frac{\sqrt{5}}{10(1-2c_2)}.}
\end{cases}
\end{equation}
By
\[
b_3a_{32}c_2=Y=\{\frac{1}{24}-(D_2c_2+D_3c_3)/(1-2c_3)\},
\]
one can solve to obtain $a_{32}$. By
\[
b_4(a_{42}c_2+a_{43}c_3)=X=\{D_3(1-c_3)-(\frac{1}{24}-D_2c_2)\}/(1-2c_3),
\]
one can take $a_{42}$ or $a_{43}$ as free parameters. By

\begin{align*}
b_2a_{52}c_2=D_2-b_3(1-c_3)a_{32}-b_3c_3a_{42},
\end{align*}
one can solve it to obtain $a_{52}$. The rest of the coefficients can be obtained by $SSRK(5)$ in \eqref{2.10}.

Taking $c_2$ as a free parameter, the whole construction procedure can be summarised as follows,

\begin{align*}
&(1)~c_1=0,c_6=1;\medskip\\
&(2)~c_3=\frac{1}{2}-\frac{\sqrt{5}}{10(1-2c_2)}, c_4=1-c_3, c_5=1-c_2;\medskip\\
&(3)~b_1=b_6=\frac{1}{6};\\
&(4)~b_5=b_2=\frac{-\frac{1}{12}(1-2c_3)^2}{c_3(1-c_3)-c_2(1-c_2)}, b_3=b_4=\frac{\frac{1}{12}(1-2c_2)^2}{c_3(1-c_3)-c_2(1-c_2)};\medskip\\
&(5)~Y=\frac{\frac{1}{96}(\frac{1}{5}+(1-2c_2)^2)-\frac{\sqrt{5}}{480}(1+(1-2c_2)^2)}{c_3(1-c_3)-c_2(1-c_2)};\\
&(6)~a_{32}=\frac{Y}{b_3c_2}=\frac{(\sqrt{5}-1)(\sqrt{5}(1-2c_2)^2-1)}{40c_2(1-2c_2)^2}, a_{31}=c_3-a_{32}=\frac{1}{2}-\frac{\sqrt{5}}{10(1-2c_2)}-a_{32};\\
&(7)~a_{43}=\frac{\alpha b_3}{720Y^2}=\frac{\alpha(1+\sqrt{5})^2(1+\sqrt{5}(1-2c_2)^2)}{12(\sqrt{5}(1-2c_2)^2-1)}, \alpha=\frac{1}{2};\\
&(8)~(a_{42}c_2+a_{43}c_3)=\frac{X}{b_3}=\{(1+\sqrt{5})(1+\sqrt{5}(1-2c_2)^2)\}/40(1-2c_2)^2;\\
&(9)~a_{42}=\frac{(1+\sqrt{5})(1+\sqrt{5}(1-2c_2)^2)}{40c_2(1-2c_2)^2}-\frac{a_{43}c_3}{2};\\
&(10)~a_{41}=(1-c_3)-(a_{42}+a_{43});\\
&(11)~a_{52}=\frac{\frac{1}{4}(1+(1-2c_2)^2)+\frac{(1-2c_2)^2}{(1-2c_3)^2}(1-c_3)a_{32}+c_3a_{42}}{c_2};\\
&(12)~a_{53}=b_3a_{42}/b_2, a_{54}=b_3a_{32}/b_2, a_{51}=(1-c_2)-a_{52}-b_3(a_{42}+a_{43})/b_2;\\
&(13)~a_{65}=b_2a_{21}/b_1,a_{64}=b_3a_{31}/b_1, a_{63}=b_3a_{41}/b_1,\\
&\qquad a_{62}=b_2a_{51}/b_1, a_{61}=1-(a_{62}+a_{63}+a_{64}+a_{65}).
\end{align*}

\subsection{Examples}

We next present three specific examples by following the above construction procedure.

\begin{exmp}\label{exm2.1}
Set $c_2=\frac{1}{2}(1-\frac{\sqrt{5}}{3})$. Then $c_3=\frac{1}{5}, b_2=-\frac{81}{132}, b_3=\frac{125}{132}$. Take $\alpha=\frac{1}{2}$. \medskip
\begin{equation*}
\begin{tabular}{c|cccccc}
0 & 0 & 0 & 0 & 0 & 0 & 0 \\
$\frac{1}{2}(1-\frac{\sqrt{5}}{3})$ & $\frac{1}{2}(1-\frac{\sqrt{5}}{3})$ & 0 & 0 & 0 & 0 & 0 \\
$\frac{1}{5}$ & $\frac{-2+3\sqrt{5}}{50}$ & $\frac{3(4-\sqrt{5})}{50}$ & 0 & 0 & 0 & 0 \\
$\frac{4}{5}$ & $\frac{99-8\sqrt{5}}{150}$ & $-\frac{2071+933\sqrt{5}}{1100}$ & $\frac{267+119\sqrt{5}}{12\times11}$ & 0 & 0 & 0 \\
$\frac{1}{2}(1+\frac{\sqrt{5}}{3})$ & $\frac{5+39\sqrt{5}}{81\times 4}$ & $-\frac{1218+805\sqrt{5}}{22\times27}$ & $\frac{5(2071+933\sqrt{5})}{81\times 44}$ & $-\frac{5(4-\sqrt{5})}{54}$ & 0 & 0 \\
$1$ & $-\frac{15+5\sqrt{5}}{24}$ & $-\frac{5+39\sqrt{5}}{8\times11}$ & $\frac{5(99-8\sqrt{5})}{12\times11}$ & $\frac{5(-2+3\sqrt{5})}{44}$ & $-\frac{27(3-\sqrt{5})}{44}$ & 0\\
\hline
 $ $  & $\frac{1}{6}$ & $-\frac{81}{132}$ & $\frac{125}{132}$ & $\frac{125}{132}$ & $-\frac{81}{132}$ & $\frac{1}{6}$
\end{tabular}
\end{equation*}
\end{exmp}

\begin{exmp}\label{exm2.3}
Set $c_2=\frac{1}{2}(1-\frac{2}{5}\sqrt{5})$. Then $c_3=\frac{1}{4}$, $b_2=-\frac{5}{33}$, $b_3=\frac{16}{33}$. Take $\alpha=\frac{1}{2}$. \medskip
\begin{equation*}
\begin{tabular}{c|cccccc}
0 & 0 & 0 & 0 & 0 & 0 & 0 \\
$\frac{1}{2}(1-\frac{2\sqrt{5}}{5})$ & $\frac{1}{2}(1-\frac{2\sqrt{5}}{5})$ & 0 & 0 & 0 & 0 & 0 \\
$\frac{1}{4}$ & -$\frac{3+\sqrt{5}}{16}$ & $\frac{7+\sqrt{5}}{16}$ & 0 & 0 & 0 & 0 \\
$\frac{3}{4}$ & $\frac{45+5\sqrt{5}}{48}$ & $-\frac{511+235\sqrt{5}}{48\times11}$ & $\frac{103+45\sqrt{5}}{12\times11}$ & 0 & 0 & 0 \\
$\frac{1}{2}(1+\frac{2\sqrt{5}}{5})$ & $-\frac{715+308\sqrt{5}}{15\times11}$ & $\frac{1035+278\sqrt{5}}{10\times33}$ & $\frac{511+235\sqrt{5}}{11\times15}$ & $-\frac{7+\sqrt{5}}{5}$ & 0 & 0 \\
$1$ & $-\frac{1694+726\sqrt{5}}{3\times121}$ & $\frac{1430+616\sqrt{5}}{3\times121}$ & $\frac{90+10\sqrt{5}}{3\times11}$ & $-\frac{6+2\sqrt{5}}{11}$ & $-\frac{5-2\sqrt{5}}{11}$ & 0\\
\hline
 $ $  & $\frac{1}{6}$ & $-\frac{5}{33}$ & $\frac{16}{33}$ & $\frac{16}{33}$ & $-\frac{5}{33}$ & $\frac{1}{6}$
\end{tabular}
\end{equation*}
\end{exmp}

\begin{exmp}\label{exm2.4}
Set $c_2=\frac{1}{4}$. Then $c_3=\frac{1}{2}(1-\frac{2}{\sqrt{5}})$, $b_2=\frac{16}{33}$, $b_3=-\frac{5}{33}$. Take $\alpha=\frac{1}{2}$. \medskip
\begin{equation*}
\begin{tabular}{c|cccccc}
0 & 0 & 0 & 0 & 0 & 0 & 0 \\
$\frac{1}{4}$ & $\frac{1}{4}$ & 0 & 0 & 0 & 0 & 0 \\
$\frac{1}{2}(1-\frac{2\sqrt{5}}{5})$ & $\frac{-4+3\sqrt{5}}{10}$ & $\frac{9-5\sqrt{5}}{10}$ & 0 & 0 & 0 & 0 \\
$\frac{1}{2}(1+\frac{2\sqrt{5}}{5})$ & $\frac{11-\sqrt{5}}{60}$ & $\frac{181+92\sqrt{5}}{11\times15}$ & $-\frac{103+45\sqrt{5}}{12\times11}$ & 0 & 0 & 0 \\
$\frac{3}{4}$ & $\frac{19+3\sqrt{5}}{96}$ & $\frac{621-7\sqrt{5}}{11\times48}$ & -$\frac{181+92\sqrt{5}}{11\times48}$ & $-\frac{9-5\sqrt{5}}{32}$ & 0 & 0 \\
$1$ & $\frac{-3+\sqrt{5}}{6}$ & $\frac{19+3\sqrt{5}}{33}$ & $-\frac{11-\sqrt{5}}{66}$ & $\frac{4-3\sqrt{5}}{11}$ & $\frac{8}{11}$ & 0\\
\hline
 $ $  & $\frac{1}{6}$ & $\frac{16}{33}$ & $-\frac{5}{33}$ & $-\frac{5}{33}$ & $\frac{16}{33}$ & $\frac{1}{6}$
\end{tabular}
\end{equation*}
\end{exmp}

\subsection{Numerical experiments}
In this subsection, we conduct some numerical experiments to verify the orders of the explicit Runge-Kutta methods constructed in the previous subsection. To that end, we consider the following nonlinear problem (\cite{Butcher2008}, p.~57),
\begin{equation}\label{nonlineareq1-butcher}
\begin{cases}
\displaystyle{\frac{dy_1}{dt}=y_3,}\medskip\\
\displaystyle{\frac{dy_2}{dt}=y_4,}\medskip\\
\displaystyle{\frac{dy_3}{dt}=-\frac{y_1}{(y_1^2+y_2^2)^{3/2}}},\medskip\\
\displaystyle{\frac{dy_4}{dt}=-\frac{y_2}{(y_1^2+y_2^2)^{3/2}}},\medskip\\
y(0)=(1,0,0,1)^T,
\end{cases}
\end{equation}
The exact solution of \eqref{nonlineareq1-butcher} is $y(t)=(\cos t,\sin t, -\sin t, \cos t)^T$. Let RK1, RK2, RK3 denote the explicit Runge-Kutta methods in Examples \ref{exm2.1}, \ref{exm2.3} and \ref{exm2.4}, respectively. The numerical results of the three methods applied to \eqref{nonlineareq1-butcher} are listed in Table 1. We compare the exact solution and the numerical solution at the time point $T=1$ with different stepsizes $h$. The error is defined by $\mbox{error}(h)=\sqrt{\sum\limits_{i=1}^4(y_i(T)-x_i^M)^2}$, where $x_i^M$ with $M=T/h$ is the approximation of the exact solution $y_i(T)$. The order is obtained by $\log_2\frac{\mbox{error}(h)}{\mbox{error}(h/2)}$. Table 1 shows that all the three methods RK1, RK2 and RK3 are convergent with order 5.

\bigskip

\begin{center}
\begin{tabular}{|c| c c| c c| c c|}
\multicolumn{7}{c} {Table 1. Order results of RK1, RK2 and RK3}\\
\hline\\
&\multicolumn{2}{c|}{RK1}&\multicolumn{2}{c|}{RK2}&\multicolumn{2}{c|}{RK3}\\
\hline\\
$h$&error&order&error&order&error&order\\
\hline\\
0.2&1.552315e-06&&3.557650e-06&&1.116439e-06& \\
0.1&4.647329e-08&5.06&9.304931e-08&5.26&3.678888e-08&4.92 \\
0.05&1.419250e-09&5.03&2.608325e-09&5.16&1.185410e-09&4.96 \\
0.025&4.3829821e-11&5.01&7.686324e-11&5.08&3.763568e-11&4.98 \\
0.0125&1.360179e-12&5.01&2.329748e-12&5.04&1.187870e-12&4.99 \\
0.00625&4.215618e-14&5.01&7.072748e-14&5.04&3.517603e-14&5.08 \\
\hline
\end{tabular}
\end{center}

\section*{Acknowledgment}
 The work of S. Gan was supported by the NSF of China, No. 11571373, 11671405 and 91630312. The work of H. Liu was supported by the FRG and startup grants from Hong Kong Baptist University, Hong Kong RGC General Research Funds, 12302415 and 12302017. The work of Z. Shang was supported by the NSF of China, No. 11671392.

\end{document}